\documentclass[12pt,reqno,a4paper]{amsart}
\usepackage[utf8]{inputenc}
\usepackage[T1]{fontenc}
\usepackage{amssymb}
\usepackage{pgf,tikz,pgfplots}
\pgfplotsset{compat=1.12}
\usepackage{mathrsfs}
\usetikzlibrary{arrows}
\usepackage{bm}
\usepackage[alpine]{ifsym}
\usepackage{xpicture}
\usepackage{calculator}
\usepackage{graphicx}
\usepackage{enumerate}

\usepackage[a4paper,top=3.3cm,bottom=3cm,left=3cm,right=3cm,bindingoffset=5mm]{geometry}

\usepackage{color}

\usepackage[colorlinks=true]{hyperref}
\newtheorem{theorem}{Theorem}[section]
\newtheorem{lemma}[theorem]{Lemma}
\newtheorem{corollary}[theorem]{Corollary}

\theoremstyle{definition}
\newtheorem{example}[theorem]{Example}
\newtheorem{remark}[theorem]{Remark}
\newtheorem{definition}[theorem]{Definition}

\newtheorem*{A}{Theorem A}
\newtheorem*{B}{Theorem B}

\numberwithin{equation}{section}

\newcommand{\betabarra}{\bar{\beta}}

\title[Curves with prescribed multiplicities and bounded negativity]{On the degree of curves with prescribed multiplicities and bounded negativity}
\author[C. Galindo]{Carlos Galindo}
\address{Universitat Jaume I, Campus de Riu Sec, Departamento de Matem\'aticas \& Institut Universitari de Matem\`atiques i Aplicacions de Castell\'o, 12071
Caste\-ll\'on de la Plana, Spain.}\email{galindo@uji.es}  \email{cavila@uji.es} \email{callejo@uji.es}

\author[F. Monserrat]{Francisco Monserrat}
\address{Instituto Universitario de
Matem\'atica Pura y Aplicada, Universidad Polit\'ecnica de Valencia,
Camino de Vera s/n, 46022 Valencia (Spain).}
\email{framonde@mat.upv.es}

\author[C.-J. Moreno-\'Avila]{Carlos-Jes\'us Moreno-\'Avila}

\author[E. P\'erez-Callejo]{Elvira P\'erez-Callejo}


\thanks{Partially supported by MCIN/AEI/10.13039/501100011033 and by ``ERDF A way of making Europe", grants  PGC2018-096446-B-C22 and RED2018-102583-T, by MCIN/AEI/10.13039/501100011033 and by "ESF Investing in your future”, grant PRE2019-089907, as well as by Universitat Jaume I, grant UJI-B2021-02.}

\begin{document}
\begin{abstract}

We provide a lower bound on the degree of curves of the projective plane $\mathbb{P}^2$  passing through the centers of a divisorial valuation $\nu$ of $\mathbb{P}^2$ with prescribed multiplicities, and an upper bound for the Seshadri-type constant of $\nu$, $\hat{\mu}(\nu)$, constant that is crucial in the Nagata-type valuative conjecture.
We also give some results related to the bounded negativity conjecture concerning those rational surfaces having the projective plane as a relatively minimal model.

\end{abstract}

\maketitle
\section{Introduction}
\label{intro}
The goals of this article are twofold. The first one is, on the one hand, to determine a lower bound on the degree of curves of the projective plane $\mathbb{P}^2$ going through a simple chain of infinitely near points with prescribed multiplicities. On the other hand, since this chain provides a divisorial valuation $\nu$ of $\mathbb{P}^2$, we obtain an upper bound of the Seshadri-type constant $\hat{\mu}(\nu)$. The second goal consists of supplying bounds on values, proposed by Harbourne in \cite{Harb1}, that give an asymptotic approach to bounded negativity on rational surfaces whose relatively minimal model is $\mathbb{P}^2$. For proving our results, we note that any divisorial valuation of the projective plane can also be regarded as a non-positive at infinity valuation of a suitable Hirzebruch surface and use the fact that the last mentioned class of valuations, introduced in \cite{GalMonMor}, determines rational surfaces with nice geometrical properties. This last argument gives a nexus between our goals.

These goals are related to interesting open conjectures. Nagata's conjecture on linear systems states that, given $n\geq 10$ very general points $p_1,\ldots,p_n$ of the complex projective plane and $n$ non-negative integers $m_1,\ldots,m_n$, the degree $d$ of any curve $C$ of $\mathbb{P}^2$ such that ${\rm mult}_{p_i}(C)\geq m_i$ for all $i=1,\ldots,n$ satisfies $d>\frac{1}{\sqrt{n}}\sum_{i=1}^n m_i$ \cite{Nagata}. In Section \ref{sec3} we consider a similar problem but within a different framework, where the points $p_i$ form the configuration of centers of an arbitrary divisorial valuation $\nu$ centered at a closed point $p_1$ of $\mathbb{P}^2$. Then, our first main result gives a bound on the degree $d$ which depends on the prescribed multiplicities, the maximum $t(\nu)$ of the set
$\{\nu(\varphi_S)\mid \mbox{$S$ is a line of $\mathbb{P}^2$} \}$,  where $\varphi_S$ stands for  an element of the local ring ${\mathcal O}_{\mathbb{P}^2,p_1}$ giving rise to a local equation of $S$, and combinatorial data that can be obtained from the dual graph of $\nu$.  These combinatorial data are $\overline{\beta}_0(\nu)$, which is the value (by $\nu$) of a general element of the maximal ideal $\mathfrak{m}$ of ${\mathcal O}_{\mathbb{P}^2,p_1}$, and the volume $\mathrm{vol}(\nu)$ of the valuation. Specifically, our result ({\it Theorem} \ref{th1} in the paper) is the following.
\begin{A} {\it
Let $\nu$ be a divisorial valuation of $\mathbb{P}^2$ whose configuration of centers is ${\mathcal C}_{\nu} = \{p_1, \ldots, p_n\}$. Set non-negative integers $m_1,\ldots,m_n$ and let $C$ be a curve of $\mathbb{P}^2$  passing through $({\mathcal C}_{\nu},(m_i)_{i=1}^n)$ (see Definition \ref{def31}). Then
$$\deg(C)\geq \frac{1}{\betabarra_0(\nu)+\left(1+ \delta_0(\nu)\right) t(\nu)}\sum_{i=1}^n v_i m_i,$$
where $(v_i)_{i=1}^n$ is the sequence of multiplicities of $\nu$,   $\delta_0(\nu)=-1$ if $\nu$ is the $\mathfrak{m}$-adic valuation and, otherwise,
$$\delta_0(\nu):= \left\lceil\frac{{\rm vol}(\nu)^{-1}-2\betabarra_0(\nu) t(\nu)}{t(\nu)^2} \right\rceil^+,$$
$\lceil x \rceil^+$ being the ceiling of a rational number $x$ if $x\geq 0$, and $0$ otherwise.}
\end{A}
The dual graph and some equivalent combinatorial data are useful tools for the study of plane valuations and related problems (see for instance \cite{Spiv, CDG-1, CDG-2}).

A {\it Nagata-type conjecture} within the framework of valuation theory was stated in  \cite{GalMonMoy} (see also \cite{DumHarKurRoeSze}). It implies the Nagata classical conjecture and states that any very general divisorial (or irrational) valuation centered at a closed point of the projective plane whose inverse normalized volume is larger than or equal to $9$ is minimal. A divisorial valuation $\nu$ is minimal when the square of the Seshadri-type constant $\hat{\mu}(\nu)$ equals the inverse of the volume of $\nu$.  The value $\hat{\mu}(\nu)$ was introduced in greater generality in \cite{BouKurMacSze} and, in our context, $\hat{\mu}(\nu) = \lim_{d \rightarrow \infty} \frac{\mu_d(\nu)}{d}$, where $\mu_d(\nu):= \max \{ \nu(h) \;| \; h \in k[u,v] \mbox{ and } \deg(h) \leq d \}$, $\{u,v\}$ being affine coordinates around the point $p_1$. The Seshadri-type constant for divisorial valuations is, generally speaking, very difficult to compute and it can be  found in different contexts as, for instance, the computation of Newton-Okounkov bodies of surfaces (see \cite{LazMus, KavKho} as introductory papers and \cite{CilFarKurLozRoeShr,GalMonMoyNic2,GalMonMor2} for the two-dimensional case).

As a consequence of Theorem A,
in {\it Corollary} \ref{muuu} we consider any divisorial valuation $\nu$ of $\mathbb{P}^2$ and, with the above notation, we complete our first goal by  providing the following bound on the value $\hat{\mu}(\nu)$:
$$\hat{\mu}(\nu)\leq \betabarra_0(\nu)+\left(1+ \delta_0(\nu)\right) t(\nu).$$
As we explained, this bound depends on combinatorial data that can be extracted from the dual graph of the valuation $\nu$ and the value (by $\nu$) of the germ of a suitable line of $\mathbb{P}^2$. In addition, we give a sequence of divisorial valuations $(\nu_n)$ such that the sequence of bounds (of $\hat{\mu} (\nu_n)$) provided in Corollary \ref{muuu} approaches asymptotically the actual values $\hat{\mu} (\nu_n)$.

Section \ref{sec4} is devoted to prove our second goal on the  asymptotic approach to bounded negativity. The {\it bounded negativity conjecture} is a folklore conjecture which states that any smooth projective surface $Z$ over the complex numbers has bounded negativity, which means that there is a non-negative integer $b(Z)$, depending only on $Z$, such that $C^2 \geq - b(Z)$ for any reduced and irreducible curve on $Z$ \cite{Harb1,bauer2013, bauer2015, Rou, Pok-Roe}. The origin of this conjecture is unclear but it was already on the mind of important mathematicians as Artin and Enriques  (see more details in the introduction of \cite{bauer2013}). It is worth to add that the conjecture (if true) gives a partial answer to a question by Demailly \cite[Question 6.9]{Dem} and  fails in positive characteristic \cite{Harb1,bauer2013}. A close conjecture, which implies the Nagata conjecture and is implied by the SHGH conjecture \cite{Cil}, states that if $C$ is a reduced and irreducible curve on the surface obtained by blowing-up the projective plane at very general points, then  $C^2 \geq -1$ and, if $C$  satisfies $C^2 = -1$, then it is a $(-1)$-curve. A weak bounded negativity conjecture was stated in \cite{Bau2, bauer2013} and proved in \cite{Hao}; this conjecture asserts that for any smooth complex projective surface $Z$ and any integer $g$, there is a non-negative integer $b(Z,g)$, depending only on $Z$ and $g$, such that $C^2 \geq -b(Z,g)$ for any reduced curve in $Z$ with geometric genus of each irreducible component less than or equal to $g$. The bounded negativity conjecture can also be stated by replacing reduced and irreducible curves by arbitrary reduced divisors \cite{bauer2013}. Moreover, if $Z$ is a surface coming from successive blowups of a surface $X$, for obtaining very negative curves on $Z$, it is worthwhile to consider curves on $Z$ giving very singular images in $X$. In this way, several recent papers \cite{bauer2015,Pokora2,Rou, Pokora} have considered special types of transversal arrangements of curves on $X$, that is, reduced divisors whose components are smooth and intersect pairwise transversally (the intersection points provide the singularities).
In this article we take a different approach to force the emergence of singularities by considering proximity relations among infinitely near points.


Let us state our main result in Section \ref{sec4}, {\it Theorem} \ref{th2} in the paper, which considers a surface $X_{\nu}$ obtained by blowing-up at the sequence of centers of a divisorial valuation $\nu$ of the projective plane. It uses the concept of tangent line of $\nu$ introduced before Section \ref{hirze}. 
\begin{B} {\it
Let $\nu$ be a divisorial valuation of $\mathbb{P}^2$ and set $Z:=X_{\nu}$. If $C$ is an integral curve of $\mathbb{P}^2$ different from the {tangent line}  of $\nu$ (if it exists) then
$$\frac{\tilde{C}^2}{\deg(C)^2}\geq -(1+\delta_0(\nu)),$$
where $\tilde{C}$ is the strict transform of $C$ on $Z$ and $\delta_0(\nu)$ is as in Theorem A.}
\end{B}

{\it Corollary} \ref{c1} gives a more general bound which holds for curves $\tilde{C}$ on any smooth rational surface having $\mathbb{P}^2$ as a relatively minimal model.

Since no general lower bound on the self-intersection of negative curves is known, in \cite[Section 1.3]{Harb1} (see also \cite[Conjecture 3.7.1]{Bau2}) it is proposed an {\it asymptotic approach.}  Given a nef (big and nef, according to \cite{Bau2}) divisor  $F$ on $Z$, this approach asks for a lower bound on the values $C^2/(F\cdot C)^2$, where $C$ runs over the integral curves on $Z$ with $F\cdot C>0$.
In {\it Corollary} \ref{cor} we provide such a bound for smooth rational surfaces $Z$ having $\mathbb{P}^2$ as a relatively minimal model and divisors $F=L^*$, where $L^*$ is the total transform on $Z$ of a general line $L$ of $\mathbb{P}^2$. {\it Corollary} \ref{cor1} gives, in the case of surfaces $X_{\nu}$ attached to a unique divisorial valuation of $\mathbb{P}^2$, a slightly rougher bound that depends only on purely combinatorial information associated with $\nu$. Finally, {\it Corollary} \ref{cor2} shows the existence of infinite families of rational surfaces having $\mathbb{P}^2$ as a relatively minimal model, with an arbitrarily big Picard number, that share the same bound as given in Corollary \ref{cor}.

Most of the literature on this subject only considers particular cases of blowups of surfaces. However, our results 
allow blowups at any family of infinitely near points.

It only remains to add that Section \ref{sec2} is expository and is devoted to present the main concepts and tools that we will use to prove our results. All the results of this paper are valid for surfaces over an algebraically closed field $k$, without any further assumption on the characteristic.


\section{Divisorial valuations}
\label{sec2}
\subsection{Plane divisorial valuations}
\label{plane}


Let $(R,\mathfrak{m})$ be a regular local ring of dimension $2$ and $K$ its quotient field. A valuation $\nu$ of $K$ is centered at $\mathfrak{m}$ when $R\cap\mathfrak{m}_\nu=\mathfrak{m}$ where $\mathfrak{m}_\nu=\{h\in K^* \ | \ \nu(h)>0\} \cup \{0\}$ and $K^*=K \setminus \{0\}$. The set of these valuations is bijective to the set of simple sequences of point blowups starting with the blowup $\pi_1$ of Spec$R$ at the closed point $p$ defined by $\mathfrak{m}$ \cite{Spiv}. A sequence of point blowups is simple when each blowup center belongs to the exceptional divisor created by the previous blowup.

Divisorial valuations $\nu$ are those defining (and defined by) finite simple sequences of the form
\begin{equation}\label{Eq_sequencepointblowingups}
\pi: Z_n\xrightarrow{\pi_n} Z_{n-1}\rightarrow \cdots \rightarrow Z_1 \xrightarrow{\pi_1} Z_0=\text{Spec}R,
\end{equation}
where, we set $p=p_1$ and $\pi_{i+1}$, $1 \leq i \leq n-1$, is the blowup of  $Z_i$ at the unique point $p_{i+1}$ in the exceptional divisor $E_i$ created by $\pi_i$ such that $\nu$ is centered at the local ring $\mathcal{O}_{Z_i,p_{i+1}}$. A center $p_i$ precedes another center $p_j$ if $p_i$ is mapped to $p_j$ by the composition of the associated sequence of blowups. The proximity relation on the set of centers of $\nu$ (or $\pi$), $\mathcal{C}_\nu=\{p=p_1, \ldots, p_n\}$,  works as follows: $p_i\to p_j$ ($p_i$ is proximate to $p_j$) if $i>j$ and $p_i$ belongs to the strict transform of $E_j$ on $Z_{i-1}$, also represented by $E_j$. Moreover, a center $p_i$ (or an exceptional divisor $E_i$) of $\pi$ is named satellite whenever $p_i \to p_j$ for some $j<i-1$. Otherwise, points and exceptional divisors are called free.

Many interesting properties of $\nu$ are encoded in the {\it dual graph} of $\nu$, $\Gamma_\nu$, which is the labelled tree whose vertices represent the exceptional divisors $E_i$ (labelled with $i$) and whose edges join the vertices $i$ and $j$ in case  $E_i \cap E_j \not = \emptyset$.

The set ${\mathcal C}_{\nu}$ is called the \emph{configuration of centers} of $\nu$ and it can be written as a union of sets
$$
\mathcal{C}_\nu= \mathcal{C}_1 \cup \cdots \cup \mathcal{C}_{g+1},
$$
where $g=0$ when all the points in ${\mathcal C}_{\nu}$ are free and $g>0$ otherwise. In this last case $\mathcal{C}_j=\{p_{\ell_{j-1}}, \ldots, p_{\ell_j}\}$, $j=1,\ldots,g$, are such that $\ell_0:=1$, the indices of the centers in each $\mathcal{C}_j$ are consecutive, there exists $r_j\in \{\ell_{j-1}+1,\ldots,\ell_j-1\}$ such that $p_i$ is free (respectively, satellite) if $\ell_{j-1}< i \leq r_j$ (respectively, $r_j <  i \leq \ell_j$), and the last set ${\mathcal C}_{g+1}$  consists of $p_{\ell_g}$ and a sequence (possibly empty) of free points $p_{\ell_g +1},\ldots,p_n$. Notice that the centers $p_{\ell_j+1}$, $1 \leq j \leq g$, are the first free points after a block of satellite points and that $p_{\ell_g +1}$ exists only when $n \neq \ell_g$.


Denote by $\mathfrak{m}_i$ the maximal ideal of the local ring $\mathcal{O}_{Z_i,p_{i+1}}$, $0 \leq i \leq n-1$, and define $\nu(\mathfrak{m}_i):=\min\{\nu(x) \,|\, x\in\mathfrak{m}_i\setminus \{0\}\}$. The sequence of {\it Puiseux exponents} of $\nu$ is an ordered set of rational numbers $\left(\beta_j'(\nu)\right)_{j=0}^{g+1}$ such that $\beta_0'(\nu)=1$ and, for $j \in \{1, \ldots, g+1\}$, $\beta_j'(\nu)$ is defined by the continued fraction
\[
\langle a_1^j;a_2^j,\ldots,a_{s_j}^j \rangle,
\]
where $a_k^j$, $1 \leq k \leq s_j$, successively counts the number of points in $\mathcal{C}_j$ with the same value $\nu(\mathfrak{m}_i)$. Notice that $\beta_{g+1}'(\nu)$ is a positive integer.

The ordered set $(\nu(\mathfrak{m}_i))_{i=1}^n$ is called the \emph{sequence of multiplicities} of $\nu$.
Following \cite{Spiv} (or \cite{DelGalNun}), one can deduce that, for a divisorial valuation $\nu$, its dual graph $\Gamma_\nu$ and its sequence of Puiseux exponents $\left(\beta_j'(\nu)\right)_{j=0}^{g+1}$ are equivalent data.

The sequence of {\it maximal contact values} of $\nu$, $\left(\overline{\beta}_j(\nu)\right)_{j=0}^{g+1}$,
has an important role in this paper. Its values can be defined as follows:
\begin{equation}\label{Eq_betabarravaluationcurvettes}
\overline{\beta}_0(\nu)= \nu(\mathfrak{m}) \text{ and } \overline{\beta}_j(\nu) = \nu(\varphi_{r_j}), 0\leq j\leq g+1 \text{ (with }r_0:=1),
\end{equation}
where $\varphi_{r_j}$ is an analytically irreducible element of $R$ whose strict transform on $Z_{r_j}$ is non-singular and transversal to the exceptional divisor $E_{r_j}$ at a general point and $p_{r_j}$ is the last free point in $\mathcal{C}_j$. Notice that when $\mathcal{C}_{g+1} = \{p_{\ell_g}\}$, $\betabarra_{g+1} = \nu (\varphi_{\ell_g})$. This sequence generates the semigroup of values of $\nu$, $\nu(R\setminus\{0\})\cup\{0\}$. It can be computed from the sequence of Puiseux exponents of $\nu$ and conversely. The last maximal contact value, $\betabarra_{g+1}$, has the following expression in terms of the sequence of multiplicities:
\begin{equation}\label{bb}
\betabarra_{g+1}=\sum_{i=1}^n \nu(\mathfrak{m}_i)^2.
\end{equation}
 More details can be found in \cite{Spiv,DelGalNun}.

The {\it volume} of $\nu$ (see \cite{EinLazSmi}) is defined as
\begin{equation}
\label{volume}
\mathrm{vol}(\nu) := \limsup_{m \rightarrow \infty} \frac{2 \; \mathrm{length}(R/\mathcal{P}_{m})}{m^2},
\end{equation}
where $\mathcal{P}_{m}:= \{h \in R | \nu(h) \geq m\} \cup \{0\}$.
Using \cite[Section 4.7]{Cas}, one can prove that $\mathrm{vol}(\nu)= 1/ \overline{\beta}_{g+1}(\nu)$ (see \cite{GalMonMoy}).

Sometimes it is useful to consider the so-called {\it normalized valuation of $\nu$}, which is defined as $\nu^N:= (1/\overline{\beta}_0(\nu)) \nu$ and gives rise to the {\it normalized volume} $$\mathrm{vol}^N(\nu):= \mathrm{vol}(\nu^N)=\frac{\overline{\beta}_0(\nu)^2}{\overline{\beta}_{g+1}(\nu)}.$$


\subsection{Divisorial valuations of $\mathbb{P}^2$}
\label{divp2}
Let $k$ be an algebraically closed field. If $X$ is a smooth projective surface over $k$, $p$ is a closed point of $X$, and $\nu$ is a divisorial valuation of the function field of $X$, centered at the maximal ideal $\mathfrak{m}$ of  ${\mathcal O}_{X,p}$, the triple $(\nu,X,p)$ is named \emph{divisorial valuation of $X$ centered at $p$}, although most of the times we simply say that \emph{$\nu$ is a divisorial valuation of $X$}.

Throughout the paper, given a divisorial valuation $(\nu,X,p)$ of a surface $X$ as above, and a curve $C$ of $X$, we denote by $\varphi_C$ the germ at $p$ defined by the image of $C$ in $R:={\mathcal O}_{X,p}$. Abusing the notation, $\varphi_C$  also denotes an element in ${\mathcal O}_{X,p}$ defining the germ. Divisorial valuations $\nu$ of $X$ correspond one-to-one to simple complete $\mathfrak{m}$-primary ideals $\mathcal{I}_\nu$ of ${\mathcal O}_{X,p}$ \cite[Theorem 4.3]{Spiv}. Abusing the notation again, in this paper a general element of $\mathcal{I}_\nu$ is named \emph{a general element} of $\nu$. Moreover, a germ $\psi$ at $p$ is a general element of $\nu$ if it is defined by an element of ${\mathcal O}_{X,p}$ which is irreducible in the $\mathfrak{m}$-adic completion $\hat{{\mathcal O}}_{X,p}$ and its strict transform (by the sequence (\ref{Eq_sequencepointblowingups}) provided by $\nu$) meets the divisor $E_n$ at a general point \cite[Definition 7.1]{Spiv}.

Consider the projective plane  over $k$, which we denote by $\mathbb{P}^2$. Fix a closed point $p \in \mathbb{P}^2$ and consider a divisorial valuation $\nu$ of $\mathbb{P}^2$ centered at $p$.  According to the above subsection, $\nu$ defines a rational surface $X_\nu$ through the sequence $\pi$ of blowups given by $\nu$:
\begin{equation}\label{Eq_sequencepointblowingupsP}
\pi: X_{\nu}=X_n\xrightarrow{\pi_n} X_{n-1}\rightarrow \cdots \rightarrow Z_1 \xrightarrow{\pi_1} X_0= \mathbb{P}^2.
\end{equation}

The Seshadri-type constant $\hat{\mu}(\nu)$ was introduced in \cite{BouKurMacSze} for real valuations $\nu$ of projective normal varieties. In our context, if $L$ denotes a general line of  $\mathbb{P}^2$ and  $a(mL)$ is the last value of the vanishing sequence of $H^0(mL)$ along $\nu$, then $\hat{\mu}(\nu) : = \lim_{m \rightarrow \infty} m^{-1} a(mL)$.


 Consider  projective coordinates $(X:Y:Z)$ on $\mathbb{P}^2$ and, without loss of generality, set $p=(1:0:0)$. Take the open subset $U_X:=\mathrm{Spec}(k[Y/X,Z/X])$ defined by the complement of the line with equation $X=0$ and, considering indeterminates $u$ and $v$, identify Spec$(k[u,v])$ with $U_X$ via the isomorphism
$k[u,v]\rightarrow k[Y/X,Z/X]$ defined by $u\mapsto Y/X$, $v\mapsto Z/X$. Then the valuation $\nu$ can be regarded as a valuation of the quotient field of $k[u,v]$ centered at the local ring $k[u,v]_{(u,v)}$.  The value $\hat{\mu}(\nu)$ can be computed as
$$
\hat{\mu}(\nu) = \lim_{d \rightarrow \infty} \frac{\mu_d(\nu)}{d},
$$
where, for each positive integer $d$,
\[
\mu_d(\nu):= \max \{ \nu(h) \;| \; h \in k[u,v] \mbox{ and } \deg(h) \leq d \}.
\]
(See \cite[Sections 2.5 and 2.6]{GalMonMoy}).



A divisorial valuation $\nu$ of $\mathbb{P}^2$  always satisfies $\hat{\mu}(\nu) \geq \sqrt{[\mathrm{vol}(\nu)]^{-1}}$ and $\nu$ is called {\it minimal} when the equality in the above expression holds. The minimality is a key concept for the Nagata-type valuative conjecture \cite{GalMonMoy} (which implies the Nagata conjecture) and states that, if $k=\mathbb{C}$, $\nu$ is a very general valuation and $[\mathrm{vol}^N(\nu)]^{-1} \geq 9$, then $\nu$ is minimal (see \cite{DumHarKurRoeSze} for the particular case of valuations where $g=0$, or $g=1$ and are defined by a satellite divisor). We conclude this section by recalling that each non-minimal divisorial valuation admits a unique supraminimal curve, that is, a curve defined by an irreducible polynomial $h \in k[u,v]$ such that $\hat{\mu}(\nu) = \nu(h)/ \deg(h)$. A proof of this fact is given in \cite[Section 3.4]{GalMonMor2}, and in \cite{DumHarKurRoeSze} for the above mentioned particular case. These proofs were given for the case $k=\mathbb{C}$, but they are valid for any algebraically closed field $k$.

Notice that $\nu$ is the $\mathfrak{m}$-adic valuation if and only if ${\mathcal C}_{\nu}$ has only one point and that, otherwise, there exists a unique projective line $H$, which we call \emph{tangent line} of $\nu$, such that $\nu(\varphi_H)>\betabarra_0$.

\subsection{Non-positive at infinity special divisorial valuations of $\mathbb{F}_\delta$}
\label{hirze}
As above, $k$ stands for an algebraically closed field. Set $\mathbb{P}^1$ the projective line over $k$. We denote by $\mathbb{F}_\delta : =\mathbb{P}(\mathcal{O}_{\mathbb{P}^1}\oplus\mathcal{O}_{\mathbb{P}^1}(-\delta))$ the $\delta$th \emph{Hirzebruch surface over the field $k$}, where $\delta \geq 0$ is an integer.

Following \cite[Section 2.2]{Reid}, $\mathbb{F}_\delta = (\mathbb{A}_k^{2*})^2/\sim$, where $\mathbb{A}_k^{2*}:= \mathbb{A}_k^2 \setminus \{(0,0)\}$,  $\mathbb{A}_k^2$ is the affine plane over $k$ and $\sim$ is given by the action of $k^* \times k^*$, $k^*:= k \setminus \{0\}$, defined as follows:
\begin{align*}
(\lambda,\mu) (X_0,X_1;Y_0,Y_1) = (\lambda X_0,\lambda X_1; \mu Y_0,\lambda^{-\delta}\mu Y_1) \\
\mbox{and} \;\; (\lambda, \mu) \in k^* \times k^*.
\end{align*}
Notice that the affine open sets  $U_{ij}:=\mathbb{F}_\delta\setminus \textbf{V}(X_iY_j)$, $0\leq i,j\leq 1$, cover the Hirzebruch surface.

In this article we consider two prime divisors $F$ and $M$ on $\mathbb{F}_\delta$ such that $F \cdot M=1$,  $F^2=0$, $M^2=\delta$ and their classes generate the Picard group $\text{Pic}(\mathbb{F}_\delta)$. When $\delta >0$,  the unique prime divisor of $\mathbb{F}_\delta$ with self-intersection $-\delta$ (called {\it special section}) will be denoted by $M_0$, and those points on $M_0$ are named special. When $\delta =0$, $M_0$ denotes a prime divisor linearly equivalent to $M$.

Fix a point $p \in  \mathbb{F}_\delta$ and let $\nu$ be a divisorial valuation of $\mathbb{F}_{\delta}$ centered at $p$.
As mentioned before, the valuation $\nu$ determines a unique sequence of blowups and then a unique rational surface $Y_\nu$:
\begin{equation}\label{Eq_sequencepointblowingupsF}
\pi: Y_{\nu}:=Y_n\xrightarrow{\pi_n} Y_{n-1}\rightarrow \cdots \rightarrow Y_1 \xrightarrow{\pi_1} Y_0= \mathbb{F}_\delta.
\end{equation}

The following concept was introduced in \cite[Definition 3.1]{GalMonMor}.
 \begin{definition}
\label{SG}
 The valuation $\nu$ is said to be {\it special} (with respect to $\mathbb{F}_\delta$ and $p$) when one of the following conditions holds:
\begin{enumerate}
\item $\delta = 0$.
\item $\delta >0$ and $p$ is a special point.
\item $\delta >0$, $p$ is not a special point and there is no integral curve in the complete linear system $|M|$ whose strict transform on $Y_\nu$ has negative self-intersection.
\end{enumerate}
\end{definition}

We are only interested in a particular type of special divisorial valuations of $\mathbb{F}_\delta$ which was introduced in \cite[Definition 3.5]{GalMonMor}. As we will see in the forthcoming Theorem \ref{Th1_caso_especial}, these valuations give rise to rational surfaces with nice geometrical properties (see also \cite{GalMonMor-Res}). Let us recall the definition; set $F_1$ the fiber of the projection morphism $\mathbb{F}_\delta \rightarrow \mathbb{P}^1$ that goes through $p$.

\begin{definition}
\label{defnonposspecial}
A special divisorial valuation $\nu$ of  $\mathbb{F}_\delta$  is called {\it non-positive} {\it at infinity} whenever $\nu(h)\leq 0$ for all $h\in\mathcal{O}_{\mathbb{F}_\delta}(\mathbb{F}_\delta\setminus (F_1\cup M_0))$.
\end{definition}

Let $Y_{\nu}$ be the rational surface given by a special valuation $\nu$ of $\mathbb{F}_\delta$ (centered at $p$). Denote by $F^*$, $M^*$, and $E_i^*$, $1 \leq i \leq n$, the total transforms on $Y_\nu$ of the above introduced divisors $F$ and $M$, and the exceptional divisors appearing in the sequence (\ref{Eq_sequencepointblowingupsF}). Given two germs at $p$, $\varphi$ and $\varphi'$, we denote by $(\varphi,\varphi')_p$ (respectively, $\text{mult}_{p_i}(\varphi)$) the intersection multiplicity of $\varphi$ and $\varphi'$ at $p$ (respectively, the multiplicity of the strict transform of $\varphi$ at $p_i$). Consider the following divisor on $Y_{\nu}$:
\begin{equation}\label{ellanda}
\Lambda_n := a_nF^* + b_nM^* - \sum_{i=1}^n\text{mult}_{p_i}(\varphi_n)E_i^*,
\end{equation}
where $
a_n:=(\varphi_n,\varphi_{M_0})_p, \; b_n:=(\varphi_n,\varphi_{F_1})_p$ and $\varphi_n$ is a general element of the divisorial valuation defined by the divisor $E_n$. Then the following result, proved in \cite[Theorem 3.6]{GalMonMor}, holds.

\begin{theorem}\label{Th1_caso_especial}
Let $\nu$ be a special divisorial valuation of $\mathbb{F}_\delta$ centered at a point $p$.  Set $Y_\nu$ the rational surface that $\nu$ defines. Then the following conditions are equivalent:
\begin{itemize}
\item[(a)] The valuation $\nu$ is non-positive at infinity.
\item[(b)] The divisor $\Lambda_n$ is nef.
\item[(c)] The inequality $2a_nb_n+b_n^2\delta\geq [\text{\emph{vol}}(\nu)]^{-1}$ holds.
\item[(d)] The cone of curves $NE(Y_\nu)$ is generated by the classes of the strict transforms on $Y_\nu$ of the fiber passing through $p$, the special section and the irreducible exceptional divisors associated with the map $\pi$ given by $\nu$.
\end{itemize}
\end{theorem}

\section{A lower bound on the degree of a curve of $\mathbb{P}^2$ with prescribed multiplicities}
\label{sec3}


We start with a definition and a lemma which will be useful to state our first main result.

\begin{definition}\label{def31}
Let $\nu$ be a divisorial valuation of $\mathbb{P}^2$, ${\mathcal C}_{\nu}=\{p_1,\ldots,p_n\}$ its configuration of centers and $X_{\nu}$ the surface that it defines. Set non-negative integers $m_1,\ldots,m_n$. We say that a curve $C$ of $\mathbb{P}^2$ \emph{passes through} the pair $({\mathcal C}_{\nu},(m_i)_{i=1}^n)$
if $C^*-\sum_{i=1}^n m_i E_i^*\geq 0$, $C^*$ (respectively, $E_i^*$) being the total transform (pull-back) of the curve $C$ (respectively, the exceptional divisor $E_i$) by the associated sequence of blowups (\ref{Eq_sequencepointblowingupsP}).
\end{definition}

\begin{lemma}\label{bbb}
Let $\nu$ be a divisorial valuation of $\mathbb{P}^2$ and let ${\mathcal C}_{\nu}=\{p_1,\ldots,p_n\}$ be its associated configuration of centers. Fix non-negative integers $m_1,\ldots,m_n$ and
let $C$ be a curve of $\mathbb{P}^2$ passing through $({\mathcal C}_{\nu},(m_i)_{i=1}^n)$. Then
$$\nu(\varphi_C)\geq \sum_{i=1}^n v_i m_i,$$
where $(v_i)_{i=1}^n$ is the {sequence of multiplicities of $\nu$}.

\end{lemma}

\begin{proof}
Let $\psi$ be an element of the local ring ${\mathcal O}_{{\mathbb{P}^2,p_1}}$ defining a {general element} of the valuation $\nu$. By the virtual Noether formula \cite[4.1.3]{Cas} we have that
$${\nu(\varphi_C)=(\psi, \varphi_C)_{p_1}}\geq \sum_{i=1}^n {\rm mult}_{p_i}(\psi) m_i=\sum_{i=1}^n v_i m_i.$$

\end{proof}

For the reader's convenience, we recall some notation which has been introduced in Section \ref{intro}. Consider any divisorial valuation $\nu$ of $\mathbb{P}^2$ and let $p$ be the closed point of $\mathbb{P}^2$ where $\nu$ is centered. We define $t(\nu):=1$ and $\delta_0(\nu):=-1$ if $\nu$ is the $\mathfrak{m}$-adic valuation, $\mathfrak{m}$ being the maximal ideal of ${\mathcal O}_{\mathbb{P}^2,p}$. Otherwise $t(\nu):=\nu(\varphi_H)$,   $H$ being the tangent line of $\nu$,  and
$$\delta_0(\nu):= \left\lceil\frac{{\rm vol}(\nu)^{-1}-2\betabarra_0(\nu) t(\nu)}{t(\nu)^2} \right\rceil^+,$$
where $\lceil x \rceil^+$ is defined as the ceiling of a rational number $x$ if $x\geq 0$, and $0$ otherwise.

Next, we state our first main result (Theorem A in the introduction).

\begin{theorem}\label{th1}
Let $\nu$ be a divisorial valuation of $\mathbb{P}^2$ (centered at the closed point $p$) and let $X_\nu$ be the surface that $\nu$ defines.
Set non-negative integers $m_1,\ldots,m_n$ and let $C$ be a curve of $\mathbb{P}^2$  passing through $({\mathcal C}_{\nu},(m_i)_{i=1}^n)$. Then
$$\deg(C)\geq \frac{1}{\betabarra_0(\nu)+\left(1+ \delta_0(\nu)\right) t(\nu)}\sum_{i=1}^n v_i m_i,$$
where $(v_i)_{i=1}^n$ is the sequence of multiplicities of $\nu$.
\end{theorem}
\begin{proof}
If $\nu$ is the $\mathfrak{m}$-adic valuation, the result holds trivially. Otherwise, with the notation as in Subsection \ref{divp2}, and without loss of generality, we can assume the following three  conditions:
\begin{itemize}
\item[(1)] $p$ is the point $(1:0:0)\in U_X$.
\item[(2)] Using the isomorphism described in Subsection \ref{divp2}, $k[u,v]_{(u,v)}$ is identified with the local ring of $\mathbb{P}^2$ at $p$.
\item[(3)]  The local equation  of the line $H$ at $p$ is $u=0$.
\end{itemize}

Let $f(u,v)=0$ be an affine equation of the restriction of $C$ on $U_X$. The strict transform $\tilde{C}$ on $X_\nu$ is linearly equivalent to the divisor
$$\deg(C)L^*-\sum_{i=1}^n {\rm mult}_{p_i}(C) E_i^*,$$
where $L^*$ (respectively, $E_i^*$) denotes the total transform on $X_\nu$ of a general line $L$ of $\mathbb{P}^2$ (respectively, of the exceptional divisor $E_i$ created by $\pi_i$). 

Consider, for an arbitrary integer $\delta\geq 0$, the Hirzebruch surface $\mathbb{F }:=\mathbb{F}_\delta$ and  homogeneous coordinates $(X_0,X_1;Y_0,Y_1)$ as defined in Subsection \ref{hirze}. The open subset $U_{00}:=\mathrm{Spec}(k[X_1/X_0,X_0^\delta Y_1/Y_0])$ is defined by the complement of the curve on $\mathbb{F}$ with homogeneous equation $X_0\cdot Y_0=0$. We can identify Spec$(k[u,v])$ with $U_{00}$ via the isomorphism $k[u,v]\rightarrow k[X_1/X_0,X_0^\delta Y_1/Y_0]$ defined by $u\mapsto X_1/X_0$, $v\mapsto X_0^\delta Y_1/Y_0$. This isomorphism  gives rise to an isomorphism of function fields $\theta: K({\mathbb{P}^2})\rightarrow K(\mathbb{F})$ and thus the previous valuation $\nu$ of $\mathbb{P}^2$ can also be regarded as a valuation of $\mathbb{F}$ centered at the point $q:=(1,0;1,0)$. Then, its configuration of centers ${\mathcal C}_{\nu}$ becomes a configuration of infinitely near points over $\mathbb{F}$. Within this setting, $f(u,v)=0$ is the affine equation in $U_{00}$ of a curve $D$ on $\mathbb{F}$ that is linearly equivalent to a divisor $a(\delta,f)F+b(\delta,f)M$, where $a(\delta,f)$ and $b(\delta,f)$ are non-negative integers depending on $\delta$ and $f$, and $u=0$ (respectively, $v=0$) is the affine equation of the fiber $F_1$ that contains $p$ (respectively, $M_0$). By Theorem \ref{Th1_caso_especial}, the valuation $\nu$ is a \emph{non-positive} at infinity (special) valuation of $\mathbb{F}$ if and only if
\begin{equation}\label{a}
2\nu(\varphi_{M_0})\nu(\varphi_{F_1})+\nu(\varphi_{F_1})^2\delta \geq \left[{\rm vol}(\nu)\right]^{-1}.
\end{equation}
Observe that
$$\nu(\varphi_{M_0})=\betabarra_0(\nu)\;\;\mbox{and}\;\; \nu(\varphi_{F_1})=\nu(u)=t(\nu)$$
and then Inequality (\ref{a}) is true if $\delta=\delta_0(\nu)$.

In the rest of the proof, $\mathbb{F}$ will represent the Hirzebruch surface $\mathbb{F}_{\delta_0(\nu)}$ and $\nu$ will be regarded as a non-positive at infinity valuation of $\mathbb{F}$,
 $\pi: Y_\nu\rightarrow \mathbb{F}$ being the composition of the sequence of blowing-ups centered at the points of ${\mathcal C}_{\nu}$. There is no confusion if we denote the associated exceptional divisors and its total transforms on $Y_\nu$ as before, that is, $E_i$ and $E_i^*$. Then, since $\nu$ is non-positive at infinity, the divisor on $Y_\nu$
$$\Lambda_n:=\betabarra_0(\nu)F^*+t(\nu)M^*-\sum_{i=1}^n v_i E_i^*,$$ given in (\ref{ellanda}), is nef (by Theorem \ref{Th1_caso_especial}).
As a consequence, $\Lambda_n\cdot \tilde{D}\geq 0$, where $\tilde{D}$ denotes the strict transform on $Y_\nu$ of $D$ which is linearly equivalent to the divisor
$$a(\delta_0(\nu),f)F^*+b(\delta_0(\nu),f)M^*-\sum_{i=1}^n {\rm mult}_{p_i}(C) E_i^*.$$
Therefore, using the forthcoming Lemma \ref{lemanuevo},
\begin{multline*}
\left[\betabarra_0(\nu)+t(\nu)\left(1+\delta_0(\nu)\right)\right]\deg(C)-\nu(\varphi_C) \\ \geq  \betabarra_0 (\nu) \deg_v(f)+t(\nu)\deg_u(f)+t(\nu)\deg_v(f)\delta_0(\nu)- \sum_{i=1}^n v_i \; {\rm mult}_{p_i}(C) \geq \Lambda_n\cdot \tilde{D}\geq 0.
\end{multline*}
Hence

\begin{multline*}
\deg(C)\geq \frac{1}{\betabarra_0(\nu)+\left(1+ \delta_0(\nu)\right) t(\nu)}\sum_{i=1}^n v_i \; {\rm mult}_{p_i}(C) \\ \geq \frac{1}{\betabarra_0(\nu)+\left(1+ \delta_0(\nu)\right) t(\nu)}\sum_{i=1}^n v_i m_i,
\end{multline*}
where the last inequality holds by Lemma \ref{bbb}. This finishes the proof.

\end{proof}

\begin{lemma}\label{lemanuevo}
With the above notations, let $f(u,v)=0$ be the equation of a curve defined in the affine open set $U_{00}$ of $\mathbb{F}_{\delta}$. Suppose that its closure $D$ in $\mathbb{F}_{\delta}$ is linearly equivalent to $a(\delta,f)F+b(\delta,f)M$. Then $$a(\delta,f)\leq \deg_u(f)\;\;\mbox{ and }\;\; b(\delta,f)\leq \deg_v(f).$$
\end{lemma}

\begin{proof}
Let $F(X_0,X_1,Y_0,Y_1)=0$ be an homogeneous equation of $D$ and consider the set $\mathcal{M}$ of monomials $$X_0^{\alpha}X_1^{\beta}Y_0^{\gamma}Y_1^{\mu}$$ appearing in the expression of $F(X_0,X_1,Y_0, Y_1)$ with non-zero coefficient. Since $X_0$ (respectively, $Y_0$) does not divide $F(X_0,X_1,Y_0,Y_1)$ there exists a monomial in $\mathcal{M}$ with $\alpha=0$ (respectively, $\gamma=0$); hence $\deg_u(f)\geq \beta\geq \beta-\delta \mu=a(\delta,f)$
(respectively, $\deg_v(f)\geq \mu=b(\delta,f)$).

\end{proof}

Theorem \ref{th1} provides a universal upper bound on the Seshadri-type constant $\hat{\mu}(\nu)$ of a divisorial valuation $\nu$ of $\mathbb{P}^2$. This bound depends only on the dual graph and the value (by $\nu$) of the tangent line of $\nu$. Let us state the result.

\begin{corollary}\label{muuu}
For any divisorial valuation $\nu$ of $\mathbb{P}^2$ it holds that
$$\hat{\mu}(\nu)\leq \betabarra_0(\nu)+\left(1+ \delta_0(\nu)\right) t(\nu).$$

\end{corollary}

Generally speaking  it is a hard task to compute the value $\hat{\mu} (\nu)$. The next example considers a case where  $\hat{\mu} (\nu)$ can be obtained and gives a family of divisorial valuations $\{\nu_a\}_{a \in \mathbb{N}_{\geq 3}} $ such that the bounds of $\hat{\mu} (\nu_a)$ obtained in Corollary \ref{muuu} approach asymptotically the actual values $\hat{\mu} (\nu_a)$.

\begin{example}
\label{exa1}
Assume in this example that $k$ is the field of complex numbers. Let $e$ be any non-negative integer. For any  integer $a\geq 3$, let $C_a$ be the unicuspidal curve of $\mathbb{P}^2$ whose equation in the homogeneous coordinates $(X:Y:Z)$ is
\[
\frac{\left( f_1 Y + b X^{a+1} \right)^a - f_1^{a+1}}{X^{a-1}} = 0,
\]
where $f_1 = X^{a-1}Z + Y^a$ and $b\neq 0$. Notice that $C_a$ is a Tono curve of Type I with $n=a-1$ and $s=2$ (see \cite{Tono, Bobadilla}) whose degree is $a^2 +1$.

Consider the configuration of infinitely near points $\mathcal{C}=\{p_1,\ldots, p_n\}$ such that the composition $\pi:Y\rightarrow \mathbb{P}^2$ of the sequence of point blowups at $\mathcal{C}$ gives rise to a minimal embedded resolution of the singularity of $C_a$.   Now consider a sequence $q_1, \ldots,q_{s}$ of $s:=(e+1)a^4-2a^3-2a^2-a$ free infinitely near points belonging to the successive strict transforms of $C_a$ and such that $q_1$ (respectively, $q_i$) is proximate to $p_n$ (respectively, $q_{i-1}$ for $i=2,\ldots,n$). Set $\nu_a$ the divisorial valuation whose associated configuration is ${\mathcal C}_{\nu_a}={\mathcal C}\cup \{q_i \}_{i=1}^s$. The sequence of maximal contact values  of $\nu_a$ is $\bar{\beta}_{0}(\nu_a)= a^2-a$, $\bar{\beta}_{1}(\nu_a)= a^2$, $\bar{\beta}_{2}(\nu_a)= a^3+2a+1$ and $\bar{\beta}_{3}(\nu_a)=(e+2)a^4-2a^3$.

Then, on the one hand, the curve $C_a$ is a supraminimal curve of the the valuation $\nu_a$ \cite[Definition 2.3]{GalMonMoy}  because $\nu_a(\varphi_{C_a}) > \sqrt{\bar{\beta}_{3}(\nu_a)} \deg(C_a)$ which means \cite[Lemma 3.10]{GalMonMoyNic2} that
\[
\hat{\mu} (\nu_a) = \frac{\nu_a(\varphi_{C_a})}{\deg(C_a)}= \frac{(e+2)a^4-2a^3}{a^2 +1}.
\]

On the other hand, the bound in Corollary \ref{muuu} is $\bar{\beta}_{0}(\nu_a) + (e+1) t(\nu_a) = (e+2)a^2-a$ because $\delta_0(\nu_a) =e$. Therefore,
$$\lim_{a\rightarrow +\infty} \frac{\hat{\mu} (\nu_a)}{(e+2)a^2-a}=1.$$

\end{example}

\section{Asymptotic approach to bounded negativity}
\label{sec4}

In \cite[Section 1.3]{Harb1} (see also \cite[Conjecture 3.7.1]{Bau2}), it is proposed an asymptotic approach to bounded negativity. This approach consists of, given a nef (big and nef, according to \cite{Bau2}) divisor $F$ on a surface $Z$, obtaining a lower bound on $C^2/(F\cdot C)^2$ for all integral curves $C$ such that  $F\cdot C>0$.

In this section we provide such a bound for surfaces $Z$ obtained from $\mathbb{P}^2$ by a finite sequence of point blowups $\pi:Z\rightarrow \mathbb{P}^2$ and $F=L^*$, the total transform  on $Z$ of a general line $L$ of $\mathbb{P}^2$. More specifically, we bound from below the following number:
 $$\lambda_{L^*}(Z):=\inf \left\{\frac{C^2}{(C\cdot L^*)^2} \mid \mbox{$C$ is an integral curve of $Z$ such that $C\cdot L^*>0$}\right\}.$$
Notice that
$$\lambda_{L^*}(Z)= \inf \left\{ \frac{\tilde{C}^2}{\deg(C)^2}\mid \mbox{$C$ is an integral curve of $\mathbb{P}^2$}\right\},$$
where $\tilde{C}$ denotes the strict transform of $C$ on $Z$.
We also prove that surfaces $Z$ as above can be grouped into packages (with infinitely many elements) sharing the same bound for the value $\lambda_{L^*}(Z)$. Our results do not assume any condition on the characteristic of the ground field $k$.

Let us state and prove the second main result of this paper (Theorem B in the introduction) and some consequences.

\begin{theorem}\label{th2}
Let $\nu$ be a divisorial valuation of $\mathbb{P}^2$ and set $Z:=X_{\nu}$ (see (\ref{Eq_sequencepointblowingupsP})). If $C$ is an integral curve of $\mathbb{P}^2$ different from the {tangent line}  of $\nu$ (if it exists) then
$$\frac{\tilde{C}^2}{\deg(C)^2}\geq -(1+\delta_0(\nu)),$$
where $\tilde{C}$ is the strict transform of $C$ on $Z$ and $\delta_0(\nu)$ the value defined in Theorem A.
\end{theorem}

\begin{proof}
 Suppose that $\nu$ is centered at $p\in \mathbb{P}^2$. We can assume without loss of generality that $\nu$ is not the $\mathfrak{m}$-adic valuation, $\mathfrak{m}$ being the maximal ideal of ${\mathcal O}_{\mathbb{P}^2,p}$ (otherwise the bound holds trivially). Set ${\mathcal C}_{\nu}=\{p_1,\ldots,p_n\}$ the configuration of centers of $\nu$ and notice that $\tilde{C}$
 is linearly equivalent to the divisor
$$\deg(C)L^*-\sum_{i=1}^n m_i E_i^*,$$
where  $m_i$ denotes the multiplicity of the strict transform of $C$ at $p_i$, $1\leq i\leq n$.

Let $f(u,v)=0$ be an equation of the restriction of $C$ to the affine chart $U_X$ of $\mathbb{P}^2$ (see Subsection \ref{divp2}) and assume, without loss of generality, that we are under Conditions (1), (2) and (3) of the proof of Theorem \ref{th1}. Then, by the arguments (and notations) of that proof, $f(u,v)=0$ can be viewed as the equation of an affine irreducible curve in the affine chart $U_{00}$ of $\mathbb{F}:= \mathbb{F}_{\delta_0(\nu)}$. Then it is the restriction to $U_{00}$ of an integral curve $D$ on $\mathbb{F}$ that is linearly equivalent to $a(\delta_0(\nu),f)F+b(\delta_0(\nu),f) M$, the valuation $\nu$ is a non-positive at infinity valuation of $\mathbb{F}$, and we can assume that it is centered at the point with homogeneous coordinates $(1:0;1:0)$. The strict transform $\tilde{D}$ of $D$ on the surface $Y_\nu$ given in (\ref{Eq_sequencepointblowingupsF}) is linearly equivalent to the divisor
$$a(\delta_0(\nu),f)F^*+b(\delta_0(\nu),f) M^*-\sum_{i=1}^n m_i E_i^*.$$
Now we distinguish two cases:
\begin{itemize}
\item \emph{Case 1:} $\tilde{D}^2<0.$ Then, since $\tilde{D}$ is integral and non-exceptional, it holds by Theorem \ref{Th1_caso_especial} that either $\tilde{D}=\tilde{F}_1$ or $\tilde{D}=\tilde{M}_0$, which implies that $C$ has degree 1. In the first case we get a contradiction since $C$ is different from the tangent line of $\nu$. In the second case, the strict transform of $C$ passes through $p=p_1$ but not through $p_2$; hence $\tilde{C}^2=0$ and the inequality given in the statement is true.




\item \emph{Case 2:} $\tilde{D}^2\geq 0$. Then
$$2\deg_u(f) \deg_v(f)+\left[\deg_v(f)\right]^2 \delta_0(\nu) -\sum_{i=1}^n m_i^2\geq 0$$
by Lemma \ref{lemanuevo}. As a consequence
$$(\delta_0(\nu)+2) \deg(C)^2\geq 2\deg_u(f) \deg_v(f)+[\deg_v(f)]^2 \delta_0(\nu)\geq \sum_{i=1}^n m_i^2$$
and, therefore,
$$\tilde{C}^2=\deg(C)^2-\sum_{i=1}^n m_i^2 \geq -(\delta_0(\nu)+1)\deg(C)^2.$$
\end{itemize}

\end{proof}

Some of the following results consider an arbitrary finite set $V = \{\nu_1,\ldots,\nu_N\}$ of divisorial valuations of $\mathbb{P}^2$. Each valuation $\nu_i$ is equipped with a morphism $\pi_i:X_{\nu_i}\rightarrow \mathbb{P}^2$ given by the composition of the blowups at its configuration of centers ${\mathcal C}_{\nu_i}$. Set ${\mathcal C}_V := \cup_{i=1}^N{\mathcal C}_{\nu_i}$ and denote by $X_V$ the surface obtained by the composition of the blowups centered at the points of ${\mathcal C}_V$ (after a suitable identification of points). Notice that any rational surface having $\mathbb{P}^2$ as a relatively minimal model is isomorphic to $X_V$ for some set $V$ as above.

\begin{corollary}\label{c1}

Let $V= \{\nu_1,\ldots,\nu_N\}$ be a finite set of divisorial valuations of $\mathbb{P}^2$ and set $Z:=X_V$. If $C$ is an integral curve of $\mathbb{P}^2$ that is not the tangent line of $\nu_i$ (whenever it exists) for all $i=1,\ldots,N$, then
$$\frac{\tilde{C}^2}{\deg(C)^2}\geq -\sum_{i=1}^N
\delta_0(\nu_i)-2N+1,$$
where $\tilde{C}$ denotes the strict transform of $C$ on $Z$ and $\delta_0(\nu_i)$ is the value defined in Theorem A.
\end{corollary}

\begin{proof}

Notice that
$$ \frac{\tilde{C}^2}{\deg(C)^2}=1-\frac{1}{\deg(C)^2}\sum_p {\rm mult}_{p}(C)^2,$$
where $p$ runs over the set ${\mathcal C}_{\nu_1}\cup \cdots \cup {\mathcal C}_{\nu_N}$. Hence
$$\frac{\tilde{C}^2}{\deg(C)^2}\geq \sum_{i=1}^N \left(1- \frac{1}{\deg(C)^2} {\bold m}_i \right)-(N-1)=\sum_{i=1}^N \frac{(\tilde{C}^{X_{\nu_i}})^2}{\deg(C)^2}-(N-1) ,$$
where ${\bold m}_i:=\sum_{p\in {\mathcal C}_{\nu_i}} {\rm mult}_{p}(C)^2$ for all $i=1,\ldots,N$ and $\tilde{C}^{X_{\nu_i}}$ denotes the strict transform of $C$ on $X_{\nu_i}$. Then the result follows by Theorem \ref{th2}.

\end{proof}

Given a finite family $\{\nu_1,\ldots,\nu_N\}$ of divisorial valuations of $\mathbb{P}^2$, we say that the points of a subset $S$ in $\cup_{i=1}^n {\mathcal C}_{\nu_i}$ \emph{are aligned} if there exists a line of $\mathbb{P}^2$ whose strict transforms pass through the points in $S$.

\begin{corollary}\label{cor}
Let $V=\{\nu_1,\ldots,\nu_N\}$ be any finite set of divisorial valuations of $\mathbb{P}^2$ and set $Z:=X_V$. Then
$$\lambda_{L^*}(Z)\geq \min\left\{ 1-\mu,-\sum_{i=1}^N
\delta_0(\nu_i)-2N+1\right\},$$
where $\mu$ denotes the maximum cardinality of the subsets of aligned points in $\bigcup_{i=1}^N {\mathcal C}_{\nu_i}$ and $\delta_0(\nu_i)$ is defined as in Theorem A.


\end{corollary}

\begin{proof}

Let $C$ be an integral curve of $\mathbb{P}^2$. If $C$ is a line of $\mathbb{P}^2$ then its strict transform $\tilde{C}$ on $Z$ satisfies $\tilde{C}^2\geq 1-\mu$. Otherwise $\frac{\tilde{C}^2}{\deg(C)^2}\geq -\sum_{i=1}^N
\delta_0(\nu_i)-2N+1$ by Corollary \ref{c1}.

\end{proof}

As before, we give an example showing the asymptotic sharpness of our bound in some cases.

\begin{example}
Assume the same situation and notations as in Example \ref{exa1}. If $\tilde{C}_a$ denotes the strict transform of the curve $C_a$ in $X_{\nu_a}$ then
$$\frac{\tilde{C}_a^2}{\deg(C_a)^2}=\frac{(a^2+1)^2-(e+2)a^4+2a^3}{(a^2+1)^2}.$$
Hence $- (e+1) \leq \lambda_{L^*}(X_{\nu_a})\leq \frac{(a^2+1)^2-(e+2)a^4+2a^3}{(a^2+1)^2}$ because $-(e+1)$ is the lower bound of $\lambda_{L^*}(X_{\nu_a})$ (for all $a\geq 3$) provided by Corollary \ref{cor}. This implies that $$\lim_{a\rightarrow +\infty} \lambda_{L^*}(X_{\nu_a})=-(e+1).$$
\end{example}

The following corollary provides, for any divisorial valuation $\nu$ of $\mathbb{P}^2$, a bound for the value $\lambda_{L^*}(X_{\nu})$ depending only on the dual graph of $\nu$ (that is, from purely combinatorial information given by the valuation $\nu$).














\begin{corollary}
\label{cor1}
Let $\nu$ be a divisorial valuation of $\mathbb{P}^2$ admitting a tangent line and with associated configuration ${\mathcal C}_{\nu}=\{p_1,\ldots,p_n\}$.

\begin{itemize}
\item[(a)] If $n\geq 3$ and $p_3$ is satellite, then
$$\lambda_{L^*}(X_{\nu})\geq  -1-\left\lceil \left(\frac{\betabarra_0(\nu)}{\betabarra_1(\nu)}\right)^2  \left[{\rm vol}^N(\nu)\right]^{-1}-2\frac{\betabarra_0(\nu)}{\betabarra_1(\nu)}\right\rceil^+.$$

\item[(b)] Otherwise, $$\lambda_{L^*}(X_{\nu})\geq \min \left\{ 1-\left\lceil \frac{\betabarra_1(\nu)}{\betabarra_0(\nu)}\right\rceil, -1-\left\lceil \frac{1}{4} \left[{\rm vol}^N(\nu)\right]^{-1}-2\frac{\betabarra_0(\nu)}{\betabarra_1(\nu)}\right\rceil^+ \right\},$$
where $\lceil\; \rceil$ denotes the ceiling function and $\lceil\; \rceil^+$ is defined as in Theorem A.
\end{itemize}
\end{corollary}

\begin{proof}

Keep the notations as in Corollary \ref{cor}. To prove (a), assume that $n\geq 3$ and $p_3$ is satellite. Then it is clear that $t(\nu)=\betabarra_1(\nu)$ and $\mu=2$. Thus, by Corollary \ref{cor},
$\lambda_{L^*}(X_{\nu})\geq -1-\delta_0(\nu)$. The fact that
$$\delta_0(\nu)=\left\lceil\frac{{\rm vol}(\nu)^{-1}-2\betabarra_0(\nu) \betabarra_1(\nu)}{\betabarra_1(\nu)^2} \right\rceil^+=\left\lceil\left(\frac{\betabarra_0(\nu)}{\betabarra_1(\nu)}\right)^2\left[{\rm vol}^N(\nu)\right]^{-1}-2\frac{\betabarra_0(\nu)}{\betabarra_1(\nu)}\right\rceil^+$$
finishes the proof in this case.

To prove (b), assume that either $n=2$ or $p_3$ is free. This implies that $2\betabarra_0(\nu)\leq t(\nu)\leq \betabarra_1(\nu)$. Then
$$\delta_0(\nu)\leq \left\lceil\frac{1}{4}\left[{\rm vol}^N(\nu)\right]^{-1}-2\frac{\betabarra_0(\nu)}{\betabarra_1(\nu)}\right\rceil^+.$$
This inequality, together with Corollary \ref{cor} and the fact that $\mu\leq \lceil \betabarra_1(\nu)/\betabarra_0(\nu) \rceil$, proves Part (b).

\end{proof}

\begin{remark}
Let $\nu$ be a divisorial valuation  of $\mathbb{P}^2$ such that $ \# {\mathcal C}_{\nu} \geq 2$ (where $ \#$ means cardinality), the bound of $\lambda_{L^*}(X_{\nu})$ provided in Corollary \ref{cor1} is not less than $1-\left\lceil\left[{\rm vol}^N(\nu)\right]^{-1}\right\rceil$. Since $\left\lceil\left[{\rm vol}^N(\nu)\right]^{-1}\right\rceil \leq \# {\mathcal C}_{\nu}$ (a consequence of Equality (\ref{bb})), the mentioned bound is not worse than the trivial bound $\lambda_{L^*}(X_{\nu})\geq 1-\# {\mathcal C}_{\nu}$. In fact one can find valuations where our bound improves the trivial one as much as one desires because, for any real number $\alpha > 1$, the set
$$
\left\{\# {\mathcal C}_{\nu}\mid \mbox{$\nu$ is a divisorial valuation of $\mathbb{P}^2$ such that $\left\lceil\left[{\rm vol}^N(\nu)\right]^{-1}\right\rceil \leq \alpha$ } \right\}
$$
is unbounded. Reasoning along the same lines, a similar statement holds for the more general bound given in Corollary \ref{cor}.
\end{remark}

We conclude this section by showing the existence of families of infinitely many rational surfaces $Z$, obtained from the projective plane by sequences of blowing-ups and with arbitrarily big Picard number, sharing the same bound for $\lambda_{L^*}(Z)$.

\begin{corollary}
\label{cor2}
Let $V= \{\nu_1,\ldots,\nu_N\}$ be any finite family of divisorial valuations of $\mathbb{P}^2$. Set $Z(1)=X_{V}$. Assume that $\nu_1,\ldots,\nu_k$ (with $1 \leq k\leq N$) admit a tangent line and that, for all $i=1,\ldots,k$, the last point $p_{n_i}$ of ${\mathcal C}_{\nu_i}=\{p_1,\ldots,p_{n_i}\}$ is free.
For each $i=1,\ldots,k$, consider any set of infinitely near points ${\mathcal D}_{\nu_i}=\{p_{n_i+1},\ldots, p_{m_i}\}$
such that, $p_{n_i+1}$ is proximate to $p_{n_i}$ and $p_{n_i-1}$ and, for all $j=n_j+2,\ldots,m_i$, $p_j$ is satellite and belongs to the first infinitesimal neighbourhood of $p_{j-1}$. Set $Z(2)=X_{V'},$
where
$$
V' = \{\nu'_1, \ldots, \nu'_k, \nu_{k+1}, \ldots, \nu_N\},
$$
$\nu'_i$, $1\leq i\leq k$, being the divisorial valuation of $\mathbb{P}^2$ whose associated configuration is ${\mathcal C}_{\nu_i}\cup {\mathcal D}_{\nu_i}$. Then $\lambda_{L^*}(Z(2))$ is not lower than the bound of $\lambda_{L^*}(Z(1))$ provided by Corollary \ref{cor}, that is,
$$\lambda_{L^*}(Z(2))\geq \min\left\{ 1-\mu,-\sum_{i=1}^N
\delta_0(\nu_i)-2N+1\right\},$$
where $\mu$ denotes the cardinal of a maximal subset of aligned points in $\bigcup_{i=1}^N {\mathcal C}_{\nu_i}$ and $\delta_0(\nu_i)$ is defined as in Theorem A.
\end{corollary}

\begin{proof}
 Pick $i\in \{1,\ldots,k\}$ and set $(\betabarra_j(\nu_i))_{j=0}^g$ the sequence of maximal contact values of the valuation $\nu_i$. Since we add satellite points, the sequence of maximal contact values of the valuation $\nu'_i$, $(\betabarra_j(\nu'_i))_{j=0}^{g+1}$, has $g+2$ elements. In addition,  defining $e_{g-1} (\nu'_i) := \gcd(\betabarra_0(\nu'_i), \betabarra_1(\nu'_i), \ldots, \betabarra_{g-1}(\nu'_i))$, by (\ref{Eq_betabarravaluationcurvettes}) it holds that
\[
\betabarra_j(\nu'_i) = e_{g-1} (\nu'_i)\betabarra_j(\nu_i), \; 0 \leq j \leq g-1,
\]
\[
\betabarra_g(\nu'_i) = e_{g-1} (\nu'_i)\betabarra_g(\nu_i) - a, \mbox{ where $a < e_{g-1}(\nu'_i)$}
\]
and
\[\betabarra_{g+1}(\nu'_i) = e_{g-1}(\nu'_i) \left(e_{g-1} (\nu'_i)\betabarra_g(\nu_i) - a\right).
 \]
 Also $t(\nu'_i) = e_{g-1}(\nu'_i) t(\nu_i)$. Then, straightforward computations prove that the difference
 \[
 \frac{\left[{\rm vol}(\nu_i)\right]^{-1}-2\betabarra_0(\nu_i) t(\nu_i)}{t(\nu_i)^2} - \frac{\left[{\rm vol}(\nu'_i)\right]^{-1}-2\betabarra_0(\nu'_i) t(\nu'_i)}{t(\nu'_i)^2}
 \]
 is positive (and less than one). This fact implies the inequality $\delta_0(\nu'_i)\leq \delta_0(\nu_i)$ which, considering Corollary \ref{cor}, concludes the proof.
\end{proof}

\section*{Acknowledgments}
We thank the anonymous reviewers for their careful reading of our manuscript. We especially thank one of the reviewers for pointing out an incorrect usage of the fibred product in a previous version of this article.


\bibliographystyle{plain}
\bibliography{BIBLIO}

\end{document}